\documentclass[reqno]{amsart}

\usepackage[utf8]{inputenc}
\usepackage{amsmath}
\numberwithin{equation}{section}
\numberwithin{figure}{section}
\numberwithin{table}{section}
\usepackage{graphicx}
\usepackage{amssymb}
\usepackage{mathrsfs}
\usepackage{amsfonts}
\usepackage[colorlinks=true, allcolors=blue]{hyperref}
\usepackage{subfigure} 
\usepackage{color}
\usepackage{extarrows}
\usepackage{booktabs}
\usepackage{hyperref}
\usepackage{tikz}
\usepackage{extarrows}
\usepackage{booktabs}
\usepackage{amsthm}
\usepackage{enumerate}
\usepackage{enumitem} 
\usepackage{indentfirst}
\usepackage{embrac}
\usepackage{cite} 
\usetikzlibrary{positioning,shapes}

\def\bZ{{\mathbb Z}}

\def\bR{{\mathbb R}}

\def\sE{{\mathscr E}}
\def\sF{{\mathscr F}}
\def\sA{{\mathscr A}}
\def\sG{{\mathscr G}}

\def\bs{\mathbf{s}}

\def\sS{\mathscr{S}}

\def\fm{\mathfrak{m}}

\def\${|\!|\!|}
\def\l|{\left|\!\left|\!\left|}
\def\r|{\right|\!\right|\!\right|}

\newtheorem{theorem}{Theorem}[section]
\newtheorem{lemma}[theorem]{Lemma}
\newtheorem{proposition}[theorem]{Proposition}
\newtheorem{corollary}[theorem]{Corollary}
\theoremstyle{definition}
\newtheorem{definition}[theorem]{Definition}

\theoremstyle{remark}
\newtheorem{remark}[theorem]{Remark}
\numberwithin{equation}{section}

\begin{document}

\title[Fukushima subspaces of quasidiffusions]{Fukushima subspaces of quasidiffusions}

\author{Liping Li}
\address{Fudan University, Shanghai, China.  }
\address{Bielefeld University,  Bielefeld, Germany.}
\email{liliping@fudan.edu.cn}
\thanks{The first named author is a member of LMNS,  Fudan University.  He is also partially supported by NSFC (No.  11931004) and Alexander von Humboldt Foundation in Germany. }

\author{Jiangang Ying}
\address{Fudan University, Shanghai, China.  }
\email{jgying@fudan.edu.cn}





\begin{abstract}
Quasidiffusions are,  by definition, time-changed Brownian motions on certain closed subset of $\mathbb{R}$.  They admit an explicit representation of Dirichlet forms in terms of so-called speed measures.  The Fukushima subspace of a Dirichlet form means another regular Dichichlet form on the same state space but having a smaller Dirichlet space.  In this paper we aim to solve the problem of Fukushima subspaces for quasidiffusions.  The main result obtains all Fukushima subspaces and characterizes their structures.  In addition,  we will also give criteria for the uniqueness of Fukushima subspaces and the existence of minimal Fukushima subspace. 

\end{abstract}

\maketitle
\tableofcontents

\section{Introduction}

Let $E$ be a locally compact separable metric space and $\fm$ be a fully supported Radon measure on $E$.  Further let $(\sE,\sF)$ and $(\sE',\sF')$ be two regular Dirichlet forms on $L^2(E,\fm)$.  (The terminologies and notations concerning Dirichlet forms are referred to in,  e.g.,  \cite{FOT11, CF12}.) We call $(\sE',\sF')$ a \emph{Fukushima subspace} of $(\sE,\sF)$ if 
\[
	\sF'\subset \sF,\quad \sE'(f,g)=\sE(f,g),\;\forall f,g\in \sF'.  
\]
In addition $(\sE',\sF')$ is called a \emph{proper} Fukushima subspace of $(\sE,\sF)$ if further $\sF'\neq \sF$.  
(Accordingly $(\sE,\sF)$ is called a \emph{Fukushima extension} of $(\sE',\sF')$; see \cite{HY22}.) The main problems concerning it are the existence of proper Fukushima subspaces and how to characterize them if exist. 

The concept of Fukushima subspace,  originally termed as \emph{regular subspace} or \emph{regular Dirichlet subspace},  was raised by Fang et al.  in \cite{FFY05}.  They investigated the problem of Fukushima subspaces for one-dimensional Brownian motion,  and the result shows that it admits infinitely many proper Fukushima subspaces and each one can be expressed in terms of certain \emph{scale function} (see Corollary~\ref{COR41}~(2)).  Afterwards Li et al.  completely solved the problem of Fukukushima subspaces and extensions for local Dirichlet forms on an interval in a series of works including \cite{LY17,  LY19,  LY19-2,  LSY20}.  To be precise,  all local Dirichlet forms on an interval have infinitely many Fukushima subspaces and extensions,  and each Fukushima subspace or extension can be represented explicitly.  

In this paper we aim to solve the problem of Fukushima subspaces for \emph{quasidiffusions}.  Quasidiffusion is by definition a \emph{time-changed Brownian motion} on certain closed subset of $\mathbb{R}$.  At a heuristic level it may be thought of as the \emph{trace} of Brownian  motion on the closed set.  This process has been widely studied in the literatures such as \cite{BK87,  KW82,  K86} and the references thereof.  
We take it into consideration because quasidiffusion is quite similar to diffusion process,  the main object of previous researches of Fukushima subspaces in,  e.g., \cite{LY17,  LY19,  LY19-2,  LSY20}.  It satisfies many properties of diffusion process like \emph{regular property} (see Definition~\ref{DEF51}) and also admits a characterization in terms of so-called \emph{scale function} and \emph{speed measure} (see \S\ref{SEC23}).  Nevertheless,  the sample paths of quasidiffusion are not continuous and may enjoy jumps between the endpoints of ``gaps" of state space.  This is usually called the \emph{skip-free property}.  By virtue of a representation theorem for so-called \emph{skip-free Hunt process},  a generalization of quasidiffusion,  initialized in \cite{L23b},  the main result Theorem~\ref{THM32} in the current paper successfully obtains all Fukushima subspaces of a quasidiffusion.   Every Fukushima subspace is a skip-free Hunt process that can be uniquely determined by a scale function belonging to a special family \eqref{eq:38-2}.  Examining \cite{FFY05},  one may find that \eqref{eq:38-2} is exactly an analogue of which characterizes all Fukushima subspaces of Brownian motion.  Closer connection between Fukushima subspaces of quaisidiffusion and those of Brownian motion will be put forward in \S\ref{SEC4}.  

Another underlying reason that we set up this work is to investigate which role the jumping part of a Dirichlet form plays in obtaining Fukushima subspaces.  The celebrated Beurling-Deny theorem (see \cite[Theorem~3.2.1]{FOT11}) decomposes every regular Dirichlet form into three parts: \emph{strongly local part},  \emph{jumping part} and \emph{killing part}.  They correspond to diffusion,  jump and killing (i.e.  jump to the ceremony) of sample paths of associated Hunt process respectively.  When considering the problem of Fukushima subspaces,  the killing part can be always ignored with the help of \emph{resurrection},  as studied in \cite[Proposition~2.13]{LY15},  and hence makes no sense.  (This is also the reason why we lose no generality to assume (QK) in \S\ref{SEC21}.) The strongly local part does make sense by examining the researches for one-dimensional case in \cite{FFY05} as well as for multi-dimensional cases in \cite{LY15,  LY16}.  However the role of jumping part seems more complicated.  As a typical example,  symmetric compound Poisson process turns to admit no proper Fukushima subspaces,  as shown in \cite[Corollary~1]{LY17b}.  Additionally the jumping part of compound Poisson type in the decomposition of a Dirichlet form can be also ignored with the help of a similar transformation to resurrection;  see \cite{LY17b}.  In other words,  jumps of compound Poisson type make no sense for obtaining Fukushima subspaces.  Due to this example and a result in \cite{LY15} stating that Fukushima subspaces must enjoy the same jumping measure as the original Dirichlet form,  we conjecture at the very start that like killing part,  jumping part can be always ignored for the problem of Fukushima subspaces.  Unfortunately this is not true.  A recent work \cite{QYZ22} disproved this conjecture because every symmetric $\alpha$-stable process on $\bR$ with $1\leq \alpha<2$ is shown to admit proper Fukushima subspaces.  (In fact,  another counterexample had been appeared earlier in \cite[Corollary~5.1]{LY19}.) At this point the role of jumping part seems variable: Simple jumps like compound Poisson processes make no sense,  while complex jumps like $\alpha$-stable  processes with $1\leq \alpha<2$ do make sense.  Regarding $\alpha$-stable processes,  $\alpha$ is a suitable parameter with the critical value $\alpha_0:=1$ that measures the complexity of jumps: They admit no proper Fukushima subspaces whenever $\alpha<\alpha_0$.  We are curious about how to measure the complexity of general jumps for the problem of Fukushima subspaces.  As a preparatory try,  quasidiffusion is a good example to examine because its jumps are more complex that compound Poisson jumps,  and the skip-free property makes them more comprehensible than stable jumps.  
The section \S\ref{SEC5} is devoted to figuring out the role of jumping part of quasidiffusion in obtaining Fukushima subspaces.  The conclusion surprises us because it is indicated that every Fukushima subspace is obtained by adjusting only the strongly local part of original Dirichlet form.  Whenever the strongly local part vanishes,  a quasidiffusion admits no proper Fukushima subspaces; see Theorem~\ref{THM54}.  In other words,  jumps of quasidiffusion belong to the class of compound Poisson jumps and $\alpha$-stable jumps with $\alpha<1$ for the problem of Fukushima subspaces.  

The rest of this paper is organized as follows.  In \S\ref{SEC2} we will review some basic facts about quasidiffusions as well as skip-free Hunt processes.  The section \S\ref{SEC3} is devoted to characterizing all Fukushima subspaces of quasidiffusions.  Then in \S\ref{SEC4} we will  investigate the connection beween Fukushima subspaces of quasidiffusion and those of Brownian motion.  In \S\ref{SEC5} we turn to explore the structure of Fukushima subspaces of a quasidiffusion.  Particularly it admits no proper Fukushima subspaces,  if and only if the state space is of zero Lebesgue measure.  It admits a unique minimal Fukushima subspace,  if and only if the state space is nowhere dense.  Finally we extend all results for quasidiffusions to those for skip-free Hunt processes in \S\ref{SEC6}.  


\subsection*{Notations}
Let $\overline{\mathbb{R}}=[-\infty, \infty]$ be the extended real number system.  A set $E\subset \overline{\bR}$ is called a \emph{nearly closed subset} of $\overline{\bR}$ if $\overline E:= E\cup \{l,r\}$ is a closed subset of $\overline{\bR}$ where $l=\inf\{x: x\in E\}$ and $r=\sup\{x: x\in E\}$.  The point $l$ or $r$ is called the left or right endpoint of $E$.  Denote by $\overline{\mathscr K}$ the family of all nearly closed subsets of $\overline{\bR}$.  Set
\[
	\mathscr K:=\{E\in \overline{\mathscr K}: E\subset \bR\},
\]
and every $E\in \mathscr K$ is called a \emph{nearly closed subset} of $\bR$.  

Let $E$ be a locally compact separable metric space.  We denote by $C(E)$ the space of all real continuous functions on $E$.  In addition,  $C_c(E)$ is the subspace of $C(E)$ consisting of all continuous functions on $E$ with compact support, and
\[
	C_\infty(E):=\{f\in C(E): \forall \varepsilon>0, \exists K\text{ compact},  |f(x)|<\varepsilon, \forall x\in E\setminus K\}.  
\]
The functions in $C_\infty(E)$ are said to be vanishing at infinity.  
Given an interval $I$,  $C_c^\infty(I)$ is the family of all smooth functions with compact support on $I$.  


\section{Review of quasidiffusions}\label{SEC2}

In this section we will review some facts about quasidiffusions,  \emph{skip-free Hunt processes} and their correspondence.  More details are referred to in \cite{L23b}.  

\subsection{Quasidiffusion}\label{SEC21}

Let $m$ be an extended real valued, right continuous,  (not necessarily strictly) increasing and non-constant function on $\bR$.  Set $m(\pm\infty):=\lim_{x\rightarrow \pm\infty}m(x)$.   Put 
\begin{equation}\label{eq:51-2}
\begin{aligned}
	&l_0:=\inf\{x\in \bR: m(x)>-\infty\},\quad r_0:=\sup\{x\in \bR: m(x)<\infty\},  \\
	&l:=\inf\{x>l_0: m(x)>m(l_0)\},\quad r:=\sup\{x<r_0: m(x)<m(r_0-)\}.  
\end{aligned}\end{equation}
To avoid trivial case assume that $l<r$.  Define
\[
E_m:=\{x\in [l,r]\cap (l_0,r_0): \exists \varepsilon>0\text{ s.t. }m(x-\varepsilon)<m(x+\varepsilon)\}.  
\]
One may verify that $E_m\in \mathscr K$ ended by $l$ and $r$; see also \cite[Lemma~3.1]{L23b}. The function $m$ corresponds to a measure on $\bR$,  still denoted by $m$ if no confusions caused.  

Let $W=(W_t, \mathscr{F}^W_t, \mathbf{P}_x)$ be a Brownian motion on $\bR$ and $\ell^W(t,x)$ be its local time normalized such that for any bounded Borel measurable function $f$ on $\bR$ and $t\geq 0$, 
\[
	\int_0^t f(W_s)ds=2\int_\bR \ell^W(t,x)f(x)dx.  
\]
Define $S_t:=\int_{\bR} \ell^W(t,x)m(dx)$ for $t\geq 0$ and
\[
	T_t:=\inf\{u>0: S_u>t\},\quad t\geq 0.  
\]
Then $T=(T_t)_{t\geq 0}$ is a strictly increasing (before $\zeta$ defined as below), right continuous family of $\sF^W_t$-stopping times with $T_0=0$,  $\mathbf{P}_x$-a.s.   Define
\[
	\sF_t:=\sF^W_{T_t},\quad \zeta:=\inf\{t>0: W_{T_t}\notin (l_0,r_0)\},\quad X_t:=W_{T_t}, \; 0\leq t<\zeta.  
\]
Then $\{X_t, \sF_t, (\mathbf{P}_x)_{x\in E_m},\zeta\}$,  called \emph{a quasidiffusion with speed measure $m$},  is a standard process with state space $E_m$ and lifetime $\zeta$; see \cite{K86}.   The definition of standard process is referred to in,  e.g.,  \cite{BG68}.  

Throughout this paper we will assume the following condition to bar the possibility of killing inside for quasidiffusions:
\begin{itemize}
\item[(QK)] If $l_0>-\infty$ (resp.  $r_0<\infty$),  then $l=l_0$  (resp.  $r=r_0$).  
\end{itemize}
More explanations are referred to in \cite[\S3]{L23b}.  

\subsection{Skip-free Hunt process}\label{SEC22}

Now fix $E\in \overline{\mathscr K}$ ended by $l$ and $r$.  We add a ceremony $\partial$ to $E$ and define $E_\partial:=E\cup \{\partial\}$.  More precisely,  $\partial$ is an additional isolated point when $E=\overline{E}$.  When $E=\overline E\setminus \{l\}$ or $E=\overline{E}\setminus \{r\}$,  $\partial$ is identified with $l$ or $r$.  When $E=\overline{E}\setminus \{l,r\}$,  $E_\partial$ is the one-point compactification of $E$. 

Let 
\[
X=\left\{\Omega,  \mathscr F_t,  \theta_t,  X_t,  \mathbf{P}_x,  \zeta\right\}
\]
be a \emph{Hunt process} on $E_\partial$,  where $\{\mathscr F_t\}_{t\in [0,\infty]}$ is the minimum augmented admissible filtration and $\zeta=\inf\{t>0: X_t=\partial\}$ is the lifetime of $X$.  
The other notations and terminologies are standard and we refer readers to,  e.g.,  \cite[Appendix~A]{CF12}. 

\begin{definition}\label{DEF51}
Let $X$ be a Hunt process on $E_\partial$.  Then $X$ is called a \emph{skip-free Hunt process} if the following are satisfied:
\begin{itemize}
\item[(SF)] \emph{Skip-free property}: $(X_{t-} \wedge X_t,  X_{t-}\vee X_t)\cap E=\emptyset$ for any $t<\zeta$,  $\mathbf{P}_x$-a.s.  and all $x\in E$.   
\item[(SR)] \emph{Regular property}: $\mathbf{P}_x(T_y<\infty)>0$ for any $x, y\in E$,  where $T_y:=\inf\{t>0: X_t=y\}$ ($\inf \emptyset:=\infty$).   
\item[(SK)] There is \emph{no killing inside} in the sense that if $\mathbf{P}_x(\zeta<\infty)>0$ for $x\in E$,  then $l$ or $r$ does not belong to $E$ and $\mathbf{P}_x(X_{\zeta-}\notin E,  \zeta<\infty)=\mathbf{P}_x(\zeta<\infty)$.  
\end{itemize}
\end{definition}

Every skip-free Hunt process admits a continuous and strictly increasing  real valued function $\bs$ on $E$,  called \emph{scale function},  and a fully supported Radon measure $\mu$ on $E$,  called \emph{speed measure}; see \cite[\S2]{L23b}.   The pair $(\bs,\mu)$ is unique up to an affine transformation.  On the contrary,  given such a pair $(\bs,\mu)$,  there exists a unique skip-free Hunt process on $E_\partial$ whose scale function is $\bs$ and speed measure is $\mu$; see \cite[Corollary~4.4]{L23b}.  

Denote by $I:=\langle l,r\rangle$ the interval ended by $l$ and $r$ such that $l\in I$ (resp.  $r\in I$) if and only if $l\in E$ (resp.  $r\in E$).  Clearly $I\setminus E$ is an open set,  and hence we may write it as a union of disjoint open intervals:
\begin{equation}\label{eq:22}
	I\setminus E=\cup_{k\geq 1} (a_k,b_k).  
\end{equation}
We extend a scale function $\bs$ on $E$ to another function $\bar{\bs}$ on $[l,r]$,  called the \emph{extended scale function},  as follows:  
\begin{equation}\label{eq:24}
	\bar{\bs}(x):=\bs(a_k)+\frac{\bs(b_k)-\bs(a_k)}{b_k-a_k}\cdot (x-a_k),\quad x\in (a_k,b_k), \;k\geq 1,
\end{equation}
and 
\[
\bar{\bs}(l):=\bs(l)=\lim_{x\downarrow l}\bs(x),\quad \bar{\bs}(r):=\bs(r)=\lim_{x\uparrow r}\bs(x).
\]	
Then Lebesgue-Stietjes measure of $\bar{\bs}$, denoted by $\lambda_{\bar{\bs}}$ or $d\bar{\bs}$,  is Radon on $I$.  Further denote by $\lambda_\bs$ or $d\bs$ the restriction of $\lambda_{\bar{\bs}}$ to $E$.  
Given $f\in C(I)$,  $f\ll \bar{\bs}$ means that $f$ is absolutely continuous with respect to $\bar{\bs}$,  i.e.  there exists an absolutely continuous function $g$ on $\bar{\bs}(I):=\{\bar{\bs}(x): x\in I\}$ such that $f=g\circ \bar{\bs}$.  Meanwhile $df/d\bar{\bs}:=g'\circ \bar{\bs}$.   Put a family of continuous functions
\[
	H^1_{e,\bs}(E):=\left\{f|_{E}: f\in C(I),  f\ll \bar{\bs},  \frac{df}{d\bar{\bs}}\in L^2(I,  \lambda_{\bar{\bs}})\right\}.  
\]
In addition for $h=f|_E\in H^1_{e,\bs}(E)$, we make the convention
\[
\frac{dh}{d\bs}:=\frac{df}{d\bar{\bs}}\bigg|_E
\]
for convenience.  Note that every $h\in  H^1_{e,\bs}(E)$ admits a finite limit $h(j):=\lim_{x\rightarrow j}h(x)$ whenever $|\bs(j)|<\infty$ for $j=l$ or $r$ (even if $j\notin E$).  The representation of Dirichlet form associated to a skip-free Hunt process is obtained in \cite[Theorem~5.3]{L23b}.  

\begin{theorem}\label{THM22}
Let $X$ be a skip-free Hunt process on $E$ whose scale function is $\bs$ and speed measure is $\mu$.  Then $X$ is symmetric with respect to $\mu$ and its associated regular Dirichlet form $(\sE,\sF)$ on $L^2(E,\mu)$ admits the following representation:
\begin{equation}\label{eq:23}
\begin{aligned}
	\sF=& \{f\in L^2(E,\mu)\cap  H^1_{e,\bs}(E): f(j)=0\\
	 &\qquad \qquad \qquad \text{whenever }j\notin E\text{ and }|\bs(j)|<\infty\text{ for }j=l\text{ or }r\},  \\
	 \sE(f,g)=&\frac{1}{2}\int_E \frac{df}{d\bs}\frac{dg}{d\bs}d\bs+\frac{1}{2}\sum_{k\geq 1}\frac{(f(b_k)-f(a_k))(g(b_k)-g(a_k))}{\bs(b_k)-\bs(a_k)},\quad f,g\in \sF.
\end{aligned}
\end{equation}
\end{theorem}

\subsection{Correspondence between quasidiffusions and skip-free Hunt processes}\label{SEC23}

The main result, Theorem~4.1,  of \cite{L23b} investigates the correspondence between quasidiffusion and skip-free Hunt process: $X$ is a quasidiffusion with speed measure $m$ satisfying (QK),  if and only if it is a skip-free Hunt process on $E=E_m$ on its natural scale,  i.e.  $\bs(x):=x$ for $x\in E$,  whose speed measure is $\mu=\frac{1}{2}m|_E$.  
Particularly,  for a general skip-free Hunt process on $E\in \overline{\mathscr K}$ with scale function $\bs$ and speed measure $\mu$,  $\tilde{X}_t:=\bs(X_t),  t\geq 0,$ forms a skip-free Hunt process on $\tilde{E}:=\bs(E)=\{\bs(x): x\in E\}\in \mathscr K$ on its natural scale whose speed measure is $\tilde{\mu}:=\mu \circ \bs^{-1}$.  Actually $\tilde{X}=(\tilde{X}_t)_{t\geq 0}$ is a quasidiffusion on $\tilde{E}$.  

\section{Fukushima subspaces of quasidiffusions}\label{SEC3}

From now on fix $E\in \mathscr K$ ended by $l$ and $r$,  and let $X=(X_t)_{t\geq 0}$ be a quasidiffusion on $E$.  Denote by $I$ the interval ended by $l$ and $r$ such that $l\in I$ (resp.  $r\in I$) if and only if $l\in E$ (resp.  $r\in E$).  The open set $I\setminus E$ is still written as \eqref{eq:22}.   In view of \S\ref{SEC23},  $X$ is a skip-free Hunt process on its natural scale.  Denote its corresponding speed measure by $\mu$.  Due to Theorem~\ref{THM22},  the associated regular Dirichlet form of $X$ on $L^2(E,\mu)$ is
\begin{equation}\label{eq:31}
\begin{aligned}
	\sF=& \{f\in L^2(E,\mu)\cap  H^1_{e}(E): f(j)=0\\
	 &\qquad \qquad \qquad \text{whenever }j\notin E\text{ and }|j|<\infty\text{ for }j=l\text{ or }r\},  \\
	 \sE(f,g)=&\frac{1}{2}\int_E f'(x)g'(x)dx+\frac{1}{2}\sum_{k\geq 1}\frac{(f(b_k)-f(a_k))(g(b_k)-g(a_k))}{b_k-a_k},\quad f,g\in \sF,
\end{aligned}
\end{equation}
where $H^1_e(E):=\{f|_E: f\in H^1_e((l,r))\}$ with
\[
H^1_e((l,r)):=\{f: f\text{ is absolutely continuous on }(l,r)\text{ and }f'\in L^2((l,r))\};
\]  
see also \cite[Theorems~3.4 and 3.6]{L23}.  

In this paper what we are concerned with is the \emph{Fukushima subspaces} of $(\sE,\sF)$.  A Fukushima subspace means another regular Dirichlet form $(\sE',\sF')$ on $L^2(E,\mu)$ such that 
\begin{equation}\label{eq:33}
	\sF'\subset \sF,\quad \sE'(f,g)=\sE(f,g),\; \forall f,g\in \sF'.  
\end{equation}
A Fukushima subspace is called a proper one,  if $\sF'\neq \sF$.  
The task is to find out all Fukushima subspaces of $(\sE,\sF)$ and to obtain their representation.  Since $(\sE',\sF')$ is regular on $L^2(E,\mu)$,  it is associated to a Hunt process, denoted by $X'$,  on $E$.  We first show that $X'$ is a skip-free Hunt process on $E$.  

\begin{lemma}\label{LM31}
Let $X'$ be the Hunt process associated to a Fukushima subspace $(\sE',\sF')$ of $(\sE,\sF)$.  Then $X'$ is a skip-free Hunt process on $E$.  
\end{lemma}
\begin{proof}
Denote by $\text{Cap}'$ and $\text{Cap}$ the $1$-capacities of $(\sE',\sF')$ and $(\sE,\sF)$ respectively.  It is easy to verify that for any Borel set $B\subset E$,  $\text{Cap}'(B)\geq \text{Cap}(B)$;  see also \cite[Remark~1.2]{LY15}.  Since every singleton of $E$ is of positive $\sE$-capacity (see,  e.g.,  \cite[Corollary~3.7]{L23}),  it follows that every singleton of $E$ is also of positive $\sE'$-capacity.   On the other hand,  \cite[Theorem~2.1]{LY15} indicates that the killing measure $k'$ of $(\sE',\sF')$ vanishes and its jumping measure $J'$ is identified with that of $(\sE,\sF)$.  More precisely,  $J'$ is supported on $\{(a_k,b_k)_2,  (b_k,a_k)_2: k\geq 1\}$,  where $(\cdot,\cdot)_2$ stands for points in $E\times E$,  and 
\[
J((a_k,b_k)_2)=J((b_k,a_k)_2)=1/(4|b_k-a_k|).
\]  
Using this fact,  one can easily verify that $(\sE',\sF')$ is \emph{strongly local-like} in the sense of \cite[Definition~5.1]{L23b}.  On account of \cite[Theorem~5.3]{L23b},  it suffices to show that $(\sE',\sF')$ is irreducible,  i.e.  any invariant set $A$ with respect to $\sE'$ satisfies $\mu(A)=0$ or $\mu(E\setminus A)=0$.  

Now take an invariant set $A$ with respect to $\sE'$.  In view of \cite[Lemma~4.6.2]{FOT11},  we may assume that $1_A$ is $\sE'$-quasi-continuous.  We first show that for any $k\geq 1$,  
\begin{equation}\label{eq:35}
a_k\in A\text{ if and only if }b_k\in A,
\end{equation}
where $a_k$ and $b_k$ appear in \eqref{eq:22}.  Argue by contradiction and suppose $a_{1}\in A$ while $b_1\notin A$.   Take an interval $(l_1,r_1)$ in the decomposition \eqref{eq:22} such that $r_1\leq a_1$ if exists, and otherwise take $(l_1,r_1)\subset E$ with $r_1\leq a_1$.   Further take another interval $(l_2,r_2)$ in the decomposition \eqref{eq:22} such that $l_2\geq b_1$ if exists, and otherwise take $(l_2,r_2)\subset E$ with $l_2\geq b_1$.  (We ignore the case that $a_1=l\in E$ or $b_1=r\in E$ because it can be treated analogously.) Using the regularity of $(\sE',\sF')$,  one may find a non-negative function $f\in \sF'\cap C_c(E)$ such that for $x\in E$,  it holds that $f(x)=1$ for $x\in [r_1,l_2]$ and $f(x)=0$ for $x\geq r_2$ or $x\leq l_1$.  We only lead to a contradiction for the case that $(l_1,r_1)$ is an interval in the decomposition \eqref{eq:22} and $(l_2, r_2)\subset E$.  (The other cases can be treated similarly.) Since $A$ is invariant,  on account of \cite[Theorem~1.6.1]{FOT11},  $f\cdot 1_A,  f\cdot 1_{A^c}\in \sF'$ and 
\begin{equation}\label{eq:32}
	\sE'(f, f)=\sE'(f\cdot 1_A, f\cdot 1_A)+\sE'(f\cdot 1_{A^c},  f\cdot 1_{A^c}).  
\end{equation}
Since \eqref{eq:33} and $(\sE,\sF)$ admits the representation \eqref{eq:31},  it follows that \eqref{eq:32} amounts  to
\begin{equation}\label{eq:34}
\sum_{k\geq 1: a_k\in A ,b_k\notin A \text{ or }a_k\notin A, b_k\in A} \frac{f(a_k)f(b_k)}{|b_k-a_k|}=0.  
\end{equation}
Since $f$ is non-negative,  $a_1\in A, b_1\notin A$ and $f(a_1)=f(b_1)=1$,  the left hand side of \eqref{eq:34} is not less than
\[
	\frac{f(a_1)f(b_1)}{|b_1-a_1|}=\frac{1}{|b_1-a_1|}>0,
\]
as leads to a contradiction.  We eventually arrive at \eqref{eq:35}.  

Next we note that $1_A$ is also $\sE$-quasi-continuous because it is $\sE'$-quasi-continuous; see \cite[Remark~1.2]{LY15}.  Particularly $t\mapsto 1_A(X_t)$ is right continuous,  
\begin{equation}\label{eq:36}
	\lim_{t'\uparrow t}1_A(X_{t'})=1_{A}(X_{t-}),\quad \forall t<\zeta
\end{equation}
and 
\[
	\mathbf{P}_x(\lim_{t'\uparrow \zeta}1_A(X_{t'})=1_A(X_{\zeta-}), X_{\zeta-}\in E)=\mathbf{P}_x(X_{\zeta-}\in E),\quad \forall x\in E.  
\]
Actually 
\begin{equation}\label{eq:37}
t\mapsto 1_A(X_t)
\end{equation}
 is continuous on $\{t<\zeta\}$.  In fact,  if $X_{t-}=X_t$,  then \eqref{eq:36} implies that \eqref{eq:37} is continuous at $t$.  Otherwise (SF) implies that $(X_{t-},X_t)_2=(a_k,b_k)_2$ or $(b_k,a_k)_2$ for some $k$.  In view of \eqref{eq:35},  we also have the continuity of \eqref{eq:37} at $t$.  Then mimicking the proof of \cite[Lemma~4.6.3]{FOT11},  we can obtain that $A$ is an invariant set with respect to $(\sE,\sF)$.  Since $(\sE,\sF)$ is irreducible,  it follows that $\mu(A)=0$ or $\mu(E\setminus A)=0$.  Therefore the irreducibility of $(\sE',\sF')$ is concluded.  That completes the proof.  
\end{proof}

Given a scale function $\bs$ on $E$,  we always denote by $\bar{\bs}$ the extended scale function on $[l,r]$ of $\bs$.  Put a family of scale functions
\begin{equation}\label{eq:38-2}
\begin{aligned}
	\mathscr S:=&\{\bs=\bar{\bs}|_E:  \bar{\bs}\text{ is strictly increasing and absolutely continuous},  \\
	&\qquad \qquad \qquad\qquad \qquad \bar{\bs}'=0\text{ or }1,  \text{a.e.  and }\bar{\bs}'(x)=1,\forall x\in I\setminus E\},
\end{aligned}\end{equation}
where $\bar{\bs}'$ is the derivative of $\bar{\bs}$.  
Due to Lemma~\ref{LM31} and Theorem~\ref{THM22},  we know that the Fukushima subspace $(\sE',\sF')$ admits the representation \eqref{eq:23} with its scale function and its speed measure.  Since $X'$ is symmetric with respect to its speed measure and its symmetrizing measures are unique up to a multiplicative constant (see \cite[\S5]{L23b}),  we may always take $\mu$ in \eqref{eq:31} to be its speed measure.   Hence each Fukushima subspace is determined by its scale function.  The following result shows that the family of all Fukushima subspaces is in a one-to-one correspondence with $\mathscr S$.  

\begin{theorem}\label{THM32}
Let $E\in\mathscr K$ and $X$ be the quasidiffusion associated to $(\sE,\sF)$  expressed as \eqref{eq:31}.  Then $(\sE',\sF')$ is a Fukushima subspace of $(\sE,\sF)$ if and only if $(\sE',\sF')$ admits the representation \eqref{eq:23} with $\bs\in \mathscr S$ and $\mu$ in \eqref{eq:31}.  
\end{theorem}
\begin{proof}
The sufficiency can be verified straightforwardly.  We only prove the necessity.  Let $(\sE',\sF')$ be a Fukushima subspace with scale function  $\bs$ and speed measure $\mu$.  We need to show $\bs\in \mathscr S$.    

Note that 
\[
	H^1_{e,c}\circ \bs:=\{f|_E: f=g\circ \bar{\bs},  g\in H^1_{e,c}(\bar{\bs}(I))\}\subset \sF',
\]
where $H^1_{e,c}(\bar{\bs}(I))$ is the family of all functions in $H^1_e((\bar{\bs}(I))$ with compact support and $\bar{\bs}(I)=\{\bar{\bs}(x): x\in I\}$.   Particularly,  $\bs$ belongs to $\sF'\subset \sF$ locally.  In view of \eqref{eq:24} and the expression \eqref{eq:31} of $\sF$,  one gets that $\bar\bs$ is locally in $H^1_e((l,r))$.  Thus $\bar{\bs}$ is absolutely continuous on $(l,r)$.  
On account of \cite[Theorem~2.1]{LY15},  $(\sE',\sF')$ has the same jumping measure as $(\sE,\sF)$.  In view of \eqref{eq:23} and \eqref{eq:31},  we have $\bs(b_k)-\bs(a_k)=b_k-a_k$ for any $k\geq 1$.  Hence \eqref{eq:24} yields that $\bar{\bs}'(x)=1$ for any $x\in \cup_{k\geq 1}(a_k,b_k)=I\setminus E$.  

It remains to show $\bar\bs'=0$ or $1$,  a.e.  on $E$.  To do this take $\varphi:=\left(g\circ \bar{\bs}\right)|_E$ with arbitrary $g\in H^1_{e,c}(\bar{\bs}(I))$.  It follows from $\sE(\varphi, \varphi)=\sE'(\varphi,\varphi)$, $\bs(b_k)-\bs(a_k)=b_k-a_k$ and the absolute continuity of $\bar{\bs}$ that 
\begin{equation}\label{eq:38}
	\int_E \left(\frac{d\varphi}{d\bar\bs}\right)^2d\bar\bs=\int_E \varphi'(x)^2dx=\int_E \left(\frac{d\varphi}{d\bar\bs}\right)^2 \bar \bs'(x)^2dx.  
\end{equation}
We argue that $\bar \bs'(x)=\bar\bs'(x)^2$,  a.e.,  so that $\bar{\bs}'=0$ or $1$,  a.e.,  can be concluded.  In fact, since $\bar{\bs}$ is locally in $H^1_e((l,r))$,  $\bar \nu(dx):=(\bar{\bs}'(x)^2-\bar{\bs}'(x))dx$ is a signed Radon measure on $(l,r)$.  Let $\nu$ be the image measure of $\bar{\nu}$ under the homeomorphism $\bar{\bs}:(l,r)\rightarrow (\bar{\bs}(l), \bar{\bs}(r))$.   It suffices to prove that the signed Radon measure $\nu$ is a zero measure on $(\bar{\bs}(l), \bar{\bs}(r))$.  Argue by contradiction and without loss of generality,  suppose $[0,3]\subset (\bar{\bs}(l), \bar{\bs}(r))$.  Let $F(y):=\nu((0,y])$ for $y\geq 0$,  which is a continuous function because $\nu$ clearly charges no singletons.  Then $F$ is not equal to $0$ at some point in $[0,3]$.  Suppose further that $F(1)=1$.  
  Note that  \eqref{eq:38} yields that
\begin{equation}\label{eq:39} 
	\int_{\bar{\bs}(l)}^{\bar{\bs}(r)} g'(y)^2 dF(y)=0,\quad \forall g\in H^1_{e,c}(\bar{\bs}(I)).  
\end{equation}
For $t\in [1,3]$,  take a function $g_t(y)=0$ for $y\notin [0,t]$,  $g_t(y)=y$ for $y\in [0,1]$ and $g_t(y):= (y-t)/(1-t)$ for $y\in [1,t]$.  Clearly $g_t\in H^1_{e,c}(\bar{\bs}(I))$, and substituting $g_t$ in \eqref{eq:39} we find that $F(t)=1-(t-1)^2$ for any $t\in [1,3]$.  Now put another function $g$ as follows: $g(y):=0$ for $y\notin [0,3]$,  $g(y):=y$ for $y\in [0,1]$,  $g(y):=1$ for $y\in [1,2]$ and $g(y):=3-y$ for $y\in [2,3]$.  We have $g\in H^1_{e,c}(\bar{\bs}(I))$ again but \eqref{eq:39} yields that
\[
	0=F(1)+(F(3)-F(2))=1+(-3-0)\neq 0,
\]
as leads to a contradiction.  That completes the proof. 
\end{proof}

\section{Extended scale functions and traces}\label{SEC4}



Let $(\sE',\sF')$ be a Fukushima subspace of $(\sE,\sF)$,  whose scale function is $\bs\in \mathscr S$.  Its extended scale function $\bar{\bs}$ belongs to
\[
\begin{aligned}
\overline{\mathscr S}:=&\{\bar{\bs}: \text{ absolutely continuous and strictly increasing on }I,  \\
&\qquad\qquad\qquad \qquad \bar{\bs}'=0\text{ or }1,  \text{a.e.  and }\bar{\bs}'(x)=1,\forall x\in I\setminus E\}. 
\end{aligned}
\]
On the contrary,  $\mathscr S$ can be obtained by restricting all functions in $\overline{\mathscr{S}}$ to $E$. 
There is a sense in which $\mathscr S$ and $\overline{\mathscr S}$ establish a closer and deeper connection.  To accomplish this,  let $W$ be a Brownian motion on $I$, which is reflecting at closed endpoints while absorbing at finite open endpoints.  In other words,  $W$ is associated to the regular Dirichlet form on $L^2(I)$: 
\[
\begin{aligned}
	&\mathscr G:=\{f\in L^2(I)\cap H^1_e((l,r)): f(j)=0\text{ whenever }j\notin I\text{ and }|j|<\infty\text{ for }j=l\text{ or }r\},  \\
	&\mathscr A(f,g):=\frac{1}{2}\int_I f'(x)g'(x)dx,\quad f,g\in \mathscr G.  
\end{aligned}\]
Then we have the following. 

\begin{corollary}\label{COR41}
Let $(\sA,\sG)$ be the Dirichlet form associated to the Brownian motion $W$ and $(\sE,\sF)$ be the Dirichlet form \eqref{eq:31} associated to the quasidiffusion $X$.  Further let $\bs\in \mathscr S$ and $\bar{\bs}\in \overline{\mathscr{S}}$ be its extended scale function.
\begin{itemize} 
\item[\rm (1)] $\bs$ determines the Fukushima subspace $(\sE',\sF')$ of $(\sE,\sF)$,  as obtained in Theorem~\ref{THM32}.  
\item[\rm (2)] $\bar{\bs}$ determines the Fukushima subspace $(\sA',\sG')$ of $(\sA,\sG)$,  where
\begin{equation}\label{eq:41-2}
\begin{aligned}
	&\sG'=\{f\in L^2(I): f\ll \bar{\bs},df/d\bar{\bs}\in L^2(I,\lambda_{\bar{\bs}}),    \\
	&\qquad\qquad\qquad  f(j)=0\text{ whenever }j\notin I\text{ and }|\bar{\bs}(j)|<\infty\text{ for }j=l\text{ or }r\},\\
	&\sA'(f,g)=\frac{1}{2}\int_I \frac{df}{d\bar{\bs}}\frac{dg}{d\bar{\bs}}d\bar{\bs},\quad f,g\in \sG'.  
\end{aligned}\end{equation}
\item[\rm (3)] $(\sE,\sF)$ is the trace Dirichlet form of $(\sA, \sG)$ on $L^2(E,\mu)$.  Accordingly $X$ is the time-changed process of $W$ by the positive continuous additive functional corresponding to $\mu$. 
\item[\rm (4)] $(\sE',\sF')$ is the trace Dirichlet form of $(\sA',\sG')$ on $L^2(E,\mu)$.  Accordingly $X'$ is the time-changed process of $W'$ by the positive continuous additive functional corresponding to $\mu$,  where $W'$ is the Hunt process associated to $(\sA',\sG')$. 
\end{itemize}
\end{corollary}
\begin{proof}
The second assertion is due to \cite{FFY05}.  The third one is indicated in  \cite[Theorem~4.1]{L23b}.  The final one can be concluded by mimicking the proof of \cite[Theorem~2.1]{LY17}.  That completes the proof. 
\end{proof}

\section{Structure of Fukushima subspaces}\label{SEC5} 

Fix a quasidiffusion $X$ on $E$ whose Dirichlet form $(\sE,\sF)$ on $L^2(E,\mu)$ is expressed as \eqref{eq:31}. 
 In this section we turn to investigate the structure of the family of all Fukushima subspaces of $(\sE,\sF)$,  i.e.  $$\mathfrak{S}:=\{\sF': (\sE',\sF')\text{ is a Fukushima subspace of }(\sE,\sF)\}.$$ 
Note that $\mathfrak{S}$ contains at least one element $\sF$ and can be obviously partially ordered by inclusion ``$\subset$".   Particularly,  if $(\sE^i,\sF^i)$,  $i=1,2$,  are Fukushima subspaces of $(\sE,\sF)$ and $\sF^1\subset \sF^2$,  then $(\sE^1,\sF^1)$ is also a Fukushima subspace of $(\sE^2,\sF^2)$.  

\subsection{Characteristic sets}

Put a family of measure-dense subsets of $I$ as follows:
\[
\overline{\mathscr H}:=\{\bar{G}\subset I: I\setminus E \subset \bar{G},  |\bar{G}\cap (c,d)|>0\text{ for any interval }(c,d)\subset I\},
\]
where $|\cdot|$ stands for the Lebesgue measure.  Set further
\[
\mathscr H:=\{\bar{G}\cap  E: \bar{G}\in \overline{\mathscr H}\}.  
\]
Note that every element in $\overline{\mathscr H}$ or $\mathscr H$ should be regarded as an a.e.  equivalence class.  For simplification we still treat them as usual sets when dealing with set operations.
Clearly there is a one-to-one correspondence between $\mathscr H$ and $\overline{\mathscr H}$.  The proof of the following lemma is truly straightforward,  so we omit it.  

\begin{lemma}\label{LM42}
\begin{itemize}
\item[\rm (1)] A function $\bs$ defined on $E$ is a scale function in $\sS$,  if and only if there is a unique set $G\in \mathscr H$ such that
\[
	\bs(x)=\bs(e)+\int_e^x 1_{G\cup (I\setminus E)}(y)dy,\quad x\in E,
\]
where $e$ is an arbitrary fixed point in $E$.  
\item[\rm (2)] A function $\bar{\bs}$ defined on $I$ belongs to $\overline{\mathscr S}$,  if and only if there is a unique set $\bar{G}\in \mathscr H$ such that
\[
	\bar\bs(x)=\bar\bs(e)+\int_e^x 1_{\bar{G}}(y)dy,\quad x\in I,
\]
where $e$ is an arbitrary fixed point in $I$.  
\item[\rm (3)] Adopt the same notations as the previous two assertions.  If $\bar{\bs}\in \overline{\mathscr S}$ is the extended scale function of $\bs\in \mathscr S$,  then  $\bar{G}=G\cup (I\setminus E)$ and $G=\bar{G}\cap E$.  
\end{itemize}
\end{lemma}

The set $G$ in this lemma for $\bs\in \mathscr S$ is actually
\[
	G:=\{x\in E: \bs'(x)=1\},  
\]
which is called the \emph{characteristic set} of $\bs$.  Accordingly 
\[
	\bar{G}=G\cup (I\setminus E)=\{x\in I: \bar{\bs}'(x)=1\}
\]
is called the \emph{characteristic set} of $\bar\bs$.  In view of Theorem~\ref{THM32} and Lemma~\ref{LM42},  the family $\mathscr H$ of characteristic sets also determines all Fukushima subspaces completely.  Hence we also call $G$ the \emph{characteristic set} of corresponding Fukushima subspace.  Characteristic sets provide an equivalent description for the partial ordering of $\mathfrak{S}$.  

\begin{proposition}\label{PRO43}
Let $\sF^i\in \mathfrak{S}$ whose characteristic set is $G_i$ for $i=1,2$.  Then $\sF^1\subset \sF^2$ if and only if $G_1\subset G_2$.  Particularly,  $\sF^1=\sF^2$ if and only if $G_1=G_2$.  
\end{proposition}
\begin{proof}
Denote the scale functions of $(\sE^i,\sF^i)$ by $\bs_i$.  The extended scale function of $\bs_i$ is denoted by $\bar{\bs}_i$.  If $G_1\subset G_2$,  then the measure $d\bar{\bs}_1$ is absolutely continuous with respect to $d\bar{\bs}_2$ and the Radon-Nikidym derivative is $d\bar{\bs}_1/d\bar{\bs}_2=0$ or $1$,  $d\bar{\bs}_2$-a.e.  It is straightforward to verify that $H^1_{e,\bs_1}(E)\subset H^1_{e,\bs_2}(E)$ and hence $\sF^1\subset \sF^2$.  To the contrary,  note that $\bs_1$ belongs to $\sF^1\subset \sF^2$ locally.  In view of the expression of $\sF^2$ (see \eqref{eq:24}),  one must have that the measure $d\bar{\bs}_1$ is absolutely continuous with respect to $d\bar{\bs}_2$.  Since $d\bar{\bs}_i(x)=1_{G_i\cup (I\setminus E)}(x)dx$ due to Lemma~\ref{LM42},  it follows that $G_1\subset G_2$.  That completes the proof.   
\end{proof}

This result readily yields the following.

\begin{corollary}\label{COR53}
Let $\sF'\in \mathfrak{S}$ whose scale function is $\bs$ and characteristic set is $G$.  Then $\sF'\neq \sF$ amounts to either of the following:
\begin{itemize}
\item[(1)] $\bs$ is not the natural scale function;
\item[(2)] $E\setminus G$ is of positive Lebesgue measure.  
\end{itemize}
\end{corollary}

\subsection{Uniqueness of Fukushima subspaces}

By means of characteristic sets we can judge whether $(\sE,\sF)$ has proper Fukushima subspaces.

\begin{theorem}\label{THM54}
Let $(\sE,\sF)$ be the Dirichlet form associated to a quasidiffusion on $E\in \mathscr K$.  Then $(\sE,\sF)$ has no proper Fukushima subspaces, if and only if $E$ is of zero Lebesgue measure.  
\end{theorem}
\begin{proof}
If $E$ is of zero Lebesgue measure,  then so is $E\setminus G$ for any $G\in \mathscr H$.  In view of Corollary~\ref{COR53},  every Fukushima subspace is not a proper one. 

To the contrary,  suppose that $E$ is of positive Lebesgue measure and we are to obtain a proper Fukushima subspace of $(\sE,\sF)$.  When $E$ has empty interior,  we put $G=\emptyset$.  It is straightforward to verify that $G\in \mathscr H$ and $E\setminus G=E$ is of positive Lebesgue measure.  Hence Corollary~\ref{COR53} yields that the Fukushima subspace with characteristic set $\emptyset$ is a proper one.  When some open interval $(c,d)\subset E$,  we take a Borel subset $G_0\subset (c,d)$ such that $0<| G_0\cap J|<|d-c|$ for any subinterval $J$ of $(c,d)$.   The existence of such $G_0$ is referred to in,  e.g.,  \cite[\S1,  Exercise~33]{F99}.  Then put $G:=G_0\cup (E\setminus (c,d))$.  Again we have $G\in \mathscr H$ and $E\setminus G$ is of positive Lebesgue measure.  On account of Corollary~\ref{COR53},  $(\sE,\sF)$ still has proper Fukushima subspaces.  That completes the proof.  
\end{proof}

\subsection{Minimal element}

We can further explore the existence of \emph{minimal element} of $\mathfrak{S}$.  Recall that $\sF'\in \mathfrak{S}$ is a \emph{mimimal} (resp.  \emph{maximal}) \emph{element} of $\mathfrak{S}$ if the only $\sF''\in \mathfrak{S}$ satisfying $\sF'\supset \sF''$ (resp.  $\sF'\subset \sF''$) is $\sF'$ itself.  Particularly,  the minimal elements of $\mathfrak{S}$, as regular Dirichlet forms,  have no proper Fukushima subspaces.   Clearly $\mathfrak{S}$ admits a unique maximal element $\sF$,  while the minimal elements may not exist.  

\begin{theorem}\label{THM55}
$\mathfrak{S}$ admits a minimal element $\sF'$,  if and only if $E$ is nowhere dense,  i.e.  $E$ has empty interior.  In the meanwhile the minimal Fukushima subspace is unique,  whose characteristic set is the empty set,  and its strongly local part vanishes,  i.e.  for any $f,g\in \sF'$,  
\begin{equation}\label{eq:41}
	\sE'(f,g)=\frac{1}{2}\sum_{k\geq 1}\frac{(f(b_k)-f(a_k))(g(b_k)-g(a_k))}{b_k-a_k}.  
\end{equation}
\end{theorem}
\begin{proof}
If $E$ is nowhere dense,  then $\emptyset \in \mathscr H$.  On account of Proposition~\ref{PRO43},  the Fukushima subspace with characteristic set $\emptyset$ is actually the unique minimal element of $\mathfrak{S}$.  Meanwhile \eqref{eq:41} can be verified straightforwardly.  

To the contrary,  let $\sF'\in \mathfrak{S}$ be a minimal element with scale function $\bs$,  and we argue that $E$ is nowhere dense by contradiction.  Suppose an open interval $(c,d)\subset E$.  Then $\bs(c)<\bs(d)$ and the open interval $(\bs(c),\bs(d))$ is a subset of $\tilde{E}:=\{\bs(x):x\in E\}$.  Note that $\tilde{X}'_t:=\bs(X'_t)$,  $t\geq 0$,  is a quasidiffusion on $\tilde{E}$.   Since $(\sE',\sF')$ has no proper Fukushima subspaces and $\bs$ is a homeomorphism between $E$ and $\tilde{E}$,  it follows from \cite[Proposition~2.5]{LY15} that the associated Dirichlet form of $\tilde{X}'$ has no proper Fukushima subspaces.  This is a contradiction of Theorem~\ref{THM54} because $\tilde{E}\supset (\bs(c),\bs(d))$ is clearly of positive Lebesgue measure.  
\end{proof}
\begin{remark}
It is worth pointing out that $E$ of zero Lebesgue measure is certainly nowhere dense.  Meanwhile $\sF$,   the unique element of $\mathfrak{S}$,  is identified with the maximal and minimal ones.  
\end{remark}

\subsection{Examples}

 Let $K\subset [0,1]$ be a generalized Cantor set (see, e.g.,  \cite[page 39]{F99}) and $\mu$ be a fully supported finite measure on $K$.  Here below are some examples that have appeared in some literatures:
 \begin{itemize}
 \item[(1)] The quasidiffusion $X$ on $E:=K$ of positive Lebesgue measure with $\mu(dx)=1_K(x)dx$: This process was raised in \cite{LY17} to study the structure of Fukushima subspaces of one-dimensional Brownian motion.  In \cite[Theorem~2.1]{LY17},  the minimal Fukushima subspace is,  in fact,  put forward.  Furthermore the Dirichlet form of $X$ enjoys a strongly local part, while the minimal Fukushima subspace does not.  
 \item[(2)] The quasidiffusion $X$ on $E:=K$ of zero Lebesgue measure with $\mu$ being,  e.g.,  the induced measure of Cantor function: It was raised in \cite[Corollay~5.1]{LY19-2} to study Fukushima extensions of one-dimensional Brownian motion.  As indicated in Theorem~\ref{THM54},  $X$ has no proper Fukushima subspaces,  while \cite[Corollay~5.1]{LY19-2} tells us that it does have proper Fukushima extensions.  
 \item[(3)] The quasidiffusion $X$ on $E:=\cup_{n\in \bZ} (K+n)$,  where $K+n:=\{x+n: x\in K\}$,  is also called a \emph{Brownian motion on Cantor set} in,  e.g.,  \cite{BEPP08}.  It admits a unique minimal Fukushima subspace due to Theorem~\ref{THM55}.  Furthermore if $K$ is of zero Lebesgue measure,  then it has no proper Fukushima subspaces.  
 \end{itemize}

Some censored examples can also be raised based on these quasidiffusions.  For example,  we put $E=K\setminus \{0\}$,  $K\setminus \{1\}$ or $K\setminus \{0,1\}$ in the first example and then obtain a quasidiffusion,  which is absorbing at $0$ or $1$.  Clearly analogical results about its Fukushima subspaces can be still reached.    

\section{Extension to skip-free Hunt processes}\label{SEC6}

Finally we give some remarks on Fukushima subspaces of skip-free Hunt processes.  Let $X$ be a skip-free Hunt process on $E\in \overline{\mathscr K}$ with scale function $\bs$ and speed measure $\mu$.  Its associated Dirichlet form $(\sE,\sF)$ is expressed as \eqref{eq:24}.   Let $I$ be the interval as in \S\ref{SEC22} and $\bar{\bs}$ be the extended scale function of $\bs$.  The crucial fact is that $\tilde{X}_t:=\bs(X_t)$,  $t \geq 0$, is a skip-free Hunt process on $\tilde{E}=\bs(E)\in \mathscr K$ on its natural scale with speed measure $\tilde{\mu}=\mu\circ \bs^{-1}$.  In other words,  $\tilde{X}$ is a quasidiffusion on $\tilde E$.   Hence repeating the arguments in the previous sections to $\tilde{X}$ and noting that $\bs: E\rightarrow \tilde{E}$ is a homeomorphism,  one may obtain analogous results for Fukushima subspaces of $(\sE,\sF)$.  We summarize them as a theorem below and omit its proof. 

\begin{theorem}
Let $(\sE,\sF)$ be as above.  Denote the family of all Fukushima subspaces of $(\sE,\sF)$ by $\mathfrak{S}_\bs:=\{\sF': (\sE',\sF')\text{ is a Fukushima subspace of }(\sE,\sF)\}$.  
\begin{itemize}
\item[\rm (1)] Every Fukushima subspace of $(\sE,\sF)$ corresponds to a skip-free Hunt process with scale function $\mathfrak{s} \in \mathscr S_\bs$ and speed measure $\mu$,  where
\[
\begin{aligned}
	\mathscr S_\bs&:=\bigg\{\mathfrak{s}=\bar{\mathfrak{s}}|_E: \bar{\mathfrak{s}}\text{ is strictly increasing on }I,  d\bar{\mathfrak{s}}\ll d\bar\bs  \\
	&\qquad\qquad\qquad \frac{d\bar{\mathfrak{s}}}{d\bar\bs}=0\text{ or }1,d\bar\bs\text{-a.e.  and }  \frac{d\bar{\mathfrak{s}}}{d\bar\bs}=1\text{ on }I\setminus E\bigg\}.  
\end{aligned}\] 
\item[\rm (2)] $\mathscr S_\bs$ corresponds to the family of characteristic sets 
\[
\mathscr H_\bs:=\left\{\bar{G}_\bs\cap E: I\setminus E\subset \bar{G}_\bs\subset I,  \lambda_{\bar{\bs}}(\bar{G}_\bs\cap (c,d))>0\text{ for }(c,d)\subset I\right\}.  
\]
More precisely,  the characteristic set of $\mathfrak{s}\in \mathscr S_\bs$ is 
\[
G_\bs=\left\{x\in E: \frac{d\bar{\mathfrak{s}}}{d\bar{\bs}}(x)=1\right\}.
\]  
\item[\rm (3)] Let $\sF'\in \mathfrak{S}_\bs$ with characteristic set $G_\bs$.  Then $\sF'\neq \sF$ if and only if $E\setminus G_\bs$ is of positive $\lambda_{\bs}$-measure.  
\item[\rm (4)] $(\sE,\sF)$ has no proper Fukushima subspaces,  if and only if $\lambda_\bs(E)=0$.
\item[\rm (5)] $\mathfrak{S}_\bs$ admits a minimal element,  if and only if $E$ is nowhere dense.  Meanwhile the minimal Fukushima subspace is unique,  whose characteristic set is the empty set,  and its strongly local part vanishes.  
\end{itemize}
\end{theorem}
\begin{remark}
We should emphasis that to be rigorous,  the characteristic set $G_\bs$ should be treated as a $\lambda_\bs$-a.e.  equivalence class.  
\end{remark}




\bibliographystyle{siam} 
\bibliography{FukSub} 

\begin{thebibliography}{10}

\bibitem{BEPP08}
{\sc S.~Bhamidi, S.~N. Evans, R.~Peled, and P.~Ralph}, {\em {Brownian motion on
  disconnected sets, basic hypergeometric functions, and some continued
  fractions of Ramanujan}}, in Probability and statistics: essays in honor of
  David A. Freedman, Inst. Math. Statist., Beachwood, OH, 2008, pp.~42--75.

\bibitem{BG68}
{\sc R.~M. Blumenthal and R.~Getoor}, {\em {Markov processes and potential
  theory}}, Pure and Applied Mathematics, Vol. 29, Academic Press, New
  York-London, 1968.

\bibitem{BK87}
{\sc G.~Burkhardt and U.~K\"uchler}, {\em {The semimartingale decomposition of
  one-dimensional quasidiffusions with natural scale}}, Stochastic Process.
  Appl., 25 (1987), pp.~237--244.

\bibitem{CF12}
{\sc Z.-Q. Chen and M.~Fukushima}, {\em {Symmetric Markov processes, time
  change, and boundary theory}}, vol.~35 of London Mathematical Society
  Monographs Series, Princeton University Press, Princeton, NJ, 2012.

\bibitem{FFY05}
{\sc X.~Fang, M.~Fukushima, and J.~Ying}, {\em {On regular Dirichlet subspaces
  of $H^{1}(I)$ and associated linear diffusions}}, Osaka J. Math., 42 (2005),
  pp.~27--41.

\bibitem{F99}
{\sc G.~B. Folland}, {\em {Real analysis}}, Pure and Applied Mathematics (New
  York), John Wiley {\&} Sons, Inc., New York, second~ed., 1999.

\bibitem{FOT11}
{\sc M.~Fukushima, Y.~Oshima, and M.~Takeda}, {\em {Dirichlet forms and
  symmetric Markov processes}}, vol.~19 of de Gruyter Studies in Mathematics,
  Walter de Gruyter {\&} Co., Berlin, extended~ed., 2011.

\bibitem{HY22}
{\sc P.~He and J.~Ying}, {\em Silverstein extension and fukushima extension},
  in Dirichlet Forms and Related Topics, Z.-Q. Chen, M.~Takeda, and T.~Uemura,
  eds., Springer Proceedings in Mathematics \& Statistics, Singapore, 2022,
  Springer Nature, p.~161–173.

\bibitem{KW82}
{\sc S.~Kotani and S.~Watanabe}, {\em {Krein's spectral theory of strings and
  generalized diffusion processes}}, in Functional analysis in Markov processes
  (Katata/Kyoto, 1981), Springer, Berlin-New York, 1982, pp.~235--259.

\bibitem{K86}
{\sc U.~K\"uchler}, {\em {On sojourn times, excursions and spectral measures
  connected with quasidiffusions}}, J. Math. Kyoto Univ., 26 (1986),
  pp.~403--421.

\bibitem{L23b}
{\sc L.~Li}, {\em {On generalization of quasidiffusions}}, in preparation.

\bibitem{L23}
\leavevmode\vrule height 2pt depth -1.6pt width 23pt, {\em On diffusions with
  discontinuous scales}, arXiv.org, math.PR (2022).

\bibitem{LSY20}
{\sc L.~Li, W.~Sun, and J.~Ying}, {\em {Effective intervals and regular
  Dirichlet subspaces}}, Stochastic Process. Appl., 130 (2020), pp.~6064--6093.

\bibitem{LY15}
{\sc L.~Li and J.~Ying}, {\em Regular subspaces of Dirichlet forms},
  Festschrift Masatoshi Fukushima, World Sci. Publ., Hackensack, NJ, 2015,
  p.~397–420.

\bibitem{LY16}
\leavevmode\vrule height 2pt depth -1.6pt width 23pt, {\em {Regular subspaces
  of skew product diffusions}}, Forum Math., 28 (2016), pp.~857--872.

\bibitem{LY17b}
\leavevmode\vrule height 2pt depth -1.6pt width 23pt, {\em {Killing transform
  on regular Dirichlet subspaces}}, Potential Anal., 46 (2017), pp.~105--118.

\bibitem{LY17}
\leavevmode\vrule height 2pt depth -1.6pt width 23pt, {\em {On structure of
  regular Dirichlet subspaces for one-dimensional Brownian motion}}, Ann.
  Probab., 45 (2017), pp.~2631--2654.

\bibitem{LY19}
\leavevmode\vrule height 2pt depth -1.6pt width 23pt, {\em {On symmetric linear
  diffusions}}, Trans. Amer. Math. Soc., 371 (2019), pp.~5841--5874.

\bibitem{LY19-2}
\leavevmode\vrule height 2pt depth -1.6pt width 23pt, {\em {Regular Dirichlet
  extensions of one-dimensional Brownian motion}}, Ann. Inst. H. Poincar\'e
  Probab. Statist., 55 (2019), pp.~1815--1849.

\bibitem{QYZ22}
{\sc D.~Qian, J.~Ying, and Y.~Zheng}, {\em Regular subspaces of symmetric
  stable processes},  (2022).
\newblock arXiv:2207.09166 [math].

\end{thebibliography}

\end{document}